\documentclass[12pt,reqno]{amsart}
\usepackage{amscd}
\usepackage{amsmath}
\usepackage{amssymb}
\usepackage{amsopn}
\usepackage{amsthm}
\usepackage{amstext}

\newtheorem{lemma}{Lemma}[section]
\newtheorem{theorem}[lemma]{Theorem}

\newtheorem{proposition}[lemma]{Proposition}
\newtheorem{corollary}[lemma]{Corollary}
\newtheorem{remark}[lemma]{Remark}

\def\A{\mathcal A}
\def\B{\mathcal B}
\def\E{\mathcal E}
\def\M{\mathrm M}
\def\U{\mathfrak U}
\def\RR{\mathrm{RR}}
\def\p{\mathcal P}

\def\diag{\mathrm{diag}}

\def\csr{\mathrm{csr}}
\def\gsr{\mathrm{gsr}}
\begin{document}
\title[Nonstable $K$--theory for extension algebras]{Nonstable $K$--theory for extension algebras
of the simple purely infinite $C^*$--algebra by
certain~$C^{*}$--algebras}
\author[Zhihua Li and Yifeng Xue]{Zhihua Li $^*$ and Yifeng Xue $^{**}$}
\thanks{Department of Mathematics and Computer Science, Yichun University, Yichun,
Jiangxi, 336000}
\thanks{Department of Mathematics, East China Normal University,
Shanghai 200241. email: yfxue@math.ecnu.edu.cn}
\thanks{Project supported by Natural Science Foundation of China (no.10771069) and
 Shanghai
\indent\ \ Leading Academic Discipline Project(no.B407)}

\baselineskip 17.3pt
\begin{abstract}
Let $0\longrightarrow
\B\stackrel{j}{\longrightarrow}E\stackrel{\pi}{\longrightarrow}\A\longrightarrow
0$ be an extension of $\A$ by $\B$, where $\A$~is a unital simple purely infinite
$C^{*}$--algebra. When $\B$ is a simple separable essential ideal of the unital $C^{*}$--algebra
$E$ with $\RR(\B)=0$ and {\rm(PC)}, $K_{0}(E)=\{[p]\mid p$ is a projection in~$E\setminus B\}$;
When $B$ is a stable $C^{*}$--algebra, $\U(C(X,E))/\U_0(C(X,E))\cong K_1(C(X,E))$ for any compact
Hausdorff space $X$.
\vspace{1mm}

\noindent {\bf Keywords}\ $K$-groups; simple purely infinite $C^*$--algebra; real rank zero.

\noindent{\bf 2000 MR Subject Classification}\ 46L05.
\end{abstract}

\maketitle

\setcounter{section}{0}
\section{Introduction}

Let $\E$ be a $C^*$--algebra. Denote by $\M_n(\E)$ the $C^*$--algebra of all $n\times n$ matrices
over $\E$. If $\E$ is unital, write $\U(\E)$ to denote the unitary group of $\E$ and $\U_0(\E)$
to denote the connected component of the unit in $\U(\E)$. Put $U(\E)=\U(\E)/\U_0(\E)$. If
$\E$ has no unit, we set $U(\E)=\U(\E^+)/\U_0(\E^+)$, where $\E^+$ is the $C^*$--algebra obtained
by adding a unit to $\E$. Two projections $p,\ q$ in $\E$ are equivalent, denoted $p\sim q$,
if $p=v^*v, q=vv^*$ for some $v\in\E$. Let $[p]$ denote the equivalence of $p$ with respect to
``$\sim$". Let $p,\ r$ be projections in $\E$. $[p]\le [r]$ (resp. $[p]<[r]$) means that there is
projection $q\le r$ (resp. $q<r$) such that $p\sim q$. A projection $p$ in $\E$ is called to be
infinite, if $[p]<[p]$. The simple $C^*$--algebra $\E$ is called to be purely infinite if every nonzero
hereditary subalgebra of $\E$ contains an infinite projection.

Let $K_0(\E)$ and $K_1(\E)$ be the $K$--groups of the $C^*$--algebra
$\E$ and let $i_\E\colon U(\E)\rightarrow K_1(\E)$ be the canonical homomorphism (cf. \cite{Bl}).

The main tasks in non--stable $K$--theory are how to use the projection in $\E$ to represent
$K_0(\E)$ and how to show $i_\E$ is isomorphic. Cuntz showed in \cite{Cu} that $K_0(\E)\cong
\{[p]\vert\,p\in\E\ \text{nonzero projection}\}$ and $i_\E$ is isomorphic,
when $\E$ is a simple unital purely infinite $C^*$--algebra. Rieffel and Xue proved that
under some restrictions of stable rank on the $C^*$--algebra $\E$, $i_\E$ may be injective,
surjective or isomorphic (cf. \cite{R1,R2}, \cite{Xue}).

Let $\B$ be a closed ideal of a unital $C^*$--algebra $E$. Let $\pi\colon E\rightarrow
E/\B=\A$ be the quotient map. We will use these symbols $E$, $\B$, $\A$ and $\pi$ throughout
the paper. Liu and Fang proved in \cite{LF} that
\begin{enumerate}
\item $K_0(E)=\{[p]\vert\,p\ \text{is a projection in}\ E\backslash\B\}$ and
\item $i_E\colon U(E)\rightarrow K_1(E)$ is isomorphic.
\end{enumerate}
when $\B=\mathcal{K}$ (the algebra of compact operators on some separable
Hilbert space) and $\A$ is a unital simple purely infinite $C^*$--algebra. Visinescu showed in
\cite{V} that the above results are  also true when $\B$ is purely infinite.

In this short note, we show that (1) is true when $\B$ is a separable simple $C^*$--algebra
with $\RR(\B)=0$ and (PC) (see \S 2 below) and $\A$ is unital simple purely infinite; We also prove
that $i_{C(X,E)}$ is isomorphic for any compact Hausdorff space $X$ when $\B$ is stable and
$\A$ is unital simple purely infinite.

\section{$K_{0}$--group of the extension algebra}

Let $\E$ be a $C^*$--algebra. $\E$ is of real rank zero, denoted by $\RR(\E)=0$, if every
self--adjoint element in $\E$ can be approximated by an self--adjoint element in $\E$ with
finite spectra (cf. \cite{BP}). A non--unital, $\sigma$--unital $C^*$--algebra $\E$ with $\RR(\E)
=0$ is said to have property (PC) if it $\E$ has finitely many (densely defined) traces, say
$\{\tau_1,\cdots,\tau_k\}$ such that following conditions are satisfied:
\begin{enumerate}
\item there is an approximate unit $\{e_n\}$ of $\E$ consisting of projections such that
$\lim\limits_{n\to\infty}\tau_i(e_n)=\infty$, $i=1,\cdots,k;$
\item for two projections $p,\,q\in\E$, if $\tau_i(p)<\tau_i(q)$, $i=1,\cdots,k$, then
$[p]\le[q]$.
\end{enumerate}

Obviously, stable simple AF--algebras with only finitely many extremal traces have (PC) and
$\A_\theta\otimes\mathcal{K}$ also has (PC), where $\A_\theta$ is the irrational rotation
algebra and $\mathcal{K}$ is the algebra of compact operators on some complex separable
Hilbert space.
\begin{remark}
{\rm Let $\E$ be a non--unital, $\sigma$--unital $C^*$--algebra with $\RR(\E)=0$ and (PC). Let
$\{f_n\}$ be an approximate unit of $\E$ consisting of increased projections. Suppose
$\lim\limits_{n\to\infty}\tau_i(e_n)=\infty$, $i=1,\cdots,k$, for some approximate unit
$\{e_n\}$ of $\E$ consisting of projections. Then there $\{e_{n_j}\}\subset\{e_n\}$ such
that $\tau_i(e_{n_j})>j$, $j\ge 1$, $i=1,\cdots,k$. Since $\lim\limits_{s\to\infty}
\|f_se_{n_j}f_s-e_{n_j}\|=0$, $j\ge 1$, we can find projections $f_{s_j}\le f_s$ for $s$ large
enough such that $f_{s_j}\sim e_{n_j}$, $j\ge 1$. Then
$$
\tau_i(f_s)\ge\tau_i(f_{s_j})=\tau_i(e_{n_j})>j,\quad i=1,\cdots,k,
$$
so that $\lim\limits_{n\to\infty}\tau_i(f_n)=\infty$, $i=1,\cdots,k$.
}
\end{remark}

With symbols as above, we can extend $\tau_i$ to $M(\E)$ by $\tau_i(x)=\sup\limits_{n\ge 1}
\tau_i(f_nxf_n)$ for positive element $x\in M(\E)$ (cf. \cite[P324]{HR}), $i=1,\cdots, k$,
where $M(\E)$ is the multiplier algebra of $\E$.
\begin{lemma}\label{ly1}
Suppose that $\B$ is an essential ideal of $E$ and $\A,\,\B$ are simple. Then every positive element
in $E\backslash\B$ is full.
\end{lemma}
\begin{proof}Let $a\in E\backslash\B$ with $a\ge 0$ and let $I(a)$ be closed ideal generated by
$a$ in $E$. Since $\pi(I(a))$ is a nonzero closed ideal in $\A$ and $\A$ is simple, we get that $1_\A\in
\pi(I(a))$ and hence there is $x\in\B$ such that $1_E+x\in I(a)$. Since $\B$ is an essential ideal,
it follows that $a\B a\not=\{0\}$. Choose a nonzero element $b\in\overline{a\B a}\subset I(a)$.
Since $\B$ is simple, $x$ is in the closed ideal of $\B$ generated by $b$. Thus, $x\in I(a)$
and consequently, $1_E\in I(a)$.
\end{proof}

The following lemma slightly improves Lemma 2.1 of \cite{V}, whose proof is essentially
same as it in \cite[Lemma 3.2]{Xue1} and \cite[Lemma 2.1]{V}.
\begin{lemma}\label{ly2}
Suppose that $\RR(\B)=0$. Let $p,\,q$ be projections in $E$ and assume that there is $v\in\A$
such that $\pi(p)=v^*v$ and $vv^*\le\pi(q)$ in $\A$. Then there is a projection $e\in p\B p$
and a partial isometry $u\in E$ such that $p-e=u^*u$, $uu^*\le q$ and $\pi(u)=v$.
\end{lemma}
\begin{proof}Let $v\in\A$ such that $\pi(p)=v^*v,\ vv^*\le\pi(q)$. Choose $u_0\in E$ such that
$\pi(u_0)=v$ and set $w=qu_0p$. Then $\pi(w^{*}w)=\pi(p),\ \pi(w)=v$. Thus, $p-w^*w\in
p\B\,p$. Since $\RR(\B)=0$, $p\B p$ has an approximate unit consisting of projections. So
there is a projection $e\in p\B p$ such that
$$
\|(p-e)(p-w^*w)(p-e)\|=\|(p-e)-(p-e)w^{*}w(p-e)\|<1.
$$

Then $z=(p-e)w^{*}w(p-e)$ is invertible in $(p-e)E(p-e)$ and $\pi(z)=\pi(p)$.
Let $s=\big((p-e)w^{*}w(p-e)\big)^{-1}$, i.e., $zs=sz=p-e$. Then $\pi(s)=\pi(p)$.
Put $u=ws^{\frac{1}{2}}$. Then $uu^*=wsw^*\le q$, $\pi(u)=v$ and
\begin{align*}
u^{*}u=&s^{\frac{1}{2}}w^{*}ws^{\frac{1}{2}}=s^{\frac{1}{2}}(p-e)w^{*}w(p-e)s^{\frac{1}{2}}\\
=&(p-e)w^{*}w(p-e)s=p-e.
\end{align*}
\end{proof}

\begin{lemma}\label{ly3}
Suppose that $\A$ is unital simple purely infinite and $\B$ is an essential ideal of a
unital $C^{*}$--algebra $E$, moreover $\B$ is separable simple with $\RR(\B)=0$ and {\rm(PC)}.
Let $p,\,q$ be projections in $E\backslash\B$ and let $r$ be a nonzero projection in $p\B p$.
Then there is a projection $r'$ in $q\B q$ such that $[r]\le [r']$.
\end{lemma}
\begin{proof}
Since $\B$ has (PC), there are densely defined traces $\tau_1,\cdots,\tau_k$ on $\B$ and
an approximate unit $\{f_n\}$ of $\B$ consisting of increased projections such that
$\lim\limits_{n\to\infty}\tau_i(f_n)=\infty$, $i=1,\cdots,k$ and $\tau_i(e)<\tau_i(f)$,
$i=1,\cdots,k$ implies that $[e]\le[f]$ for any two projections $e,\,f$ in $\B$.

By Lemma \ref{ly1}, there are $x_1,\cdots,x_m \in\B$ such that
$\sum\limits^m_{i=1}x_i^*qx_i=1_E$. We regard $E$ as a $C^*$--subalgebra of $M(\B)$ for $\B$ is
essential. Thus,
$$
\infty=\tau_i(1_E)=\sum\limits^m_{j=1}\tau_i(x^*_jqx_j)\le\sum\limits^m_{j=1}\tau_i(\|x_j\|^2q),
$$
i.e., $\tau_i(q)=\infty$, $i=1,\cdots,k$. Let $r$ be a nonzero projection in $p\B p$.
Let $\{g_n\}$ be an approximate unit for $q\B q$ consisting of increased projections. Since
$\sup\limits_{n\ge 1}\tau_i(g_n)=\tau_i(q)=\infty$, $i=1,\cdots,k$, it follows that there is
$n_0$ such that $\tau_i(g_{n_0})>\tau_i(r)$, $i=1,\cdots,k$. Put $r'=g_{n_0}$. Then we get
$[r]\le[r']$.
\end{proof}

Now we can prove the main result of the section as follows:
\begin{theorem}
Suppose that $\A$ is unital simple purely infinite and $\B$ is an essential ideal of $E$,
moreover $\B$ is separable simple with $\RR(\B)=0$ and {\rm(PC)}. Then
$$K_0(E)=\{[p]\vert\,p\ \text{is a projection in}\ E\backslash\B\}.$$
\end{theorem}
\begin{proof}
Set $\p(E)=\{p\ \text{is a projection in}\ E\backslash\B\}$.
By \cite[Theroem 1.4]{Cu}, when $\p(E)$ satisfies following conditions:
\begin{enumerate}
\item[$(\Pi_1)$] If~$p,\ q\in\p(E)$ and $pq=0$, then $p+q\in\p(E);$

\item[$(\Pi_2)$] If $p\in\p(E)$ and $p'$ is a projection in $E$ such that $p\sim p'$, then
$p'\in\p(E);$

\item[$(\Pi_{3})$] For any $p,q\in\p(E)$, there is $p'$ such that $p'\sim p,\ p'<q$ and
$q-p'\in\p(E);$

\item[$(\Pi_{4})$] If $q$ is a projection in $E$ and there is $p\in\p(E)$ such that
$p\le q$, then $p\in\p(E)$,
\end{enumerate}
then $K_0(E)=\{[p]\vert\,p\in\p(E)\}$.  Therefore, we need only check that $\p(E)$ satisfies
above conditions.

Let $\p(\A)$ be the set of all nonzero projections in $\A$. By \cite[Proposition 1.5]{Cu},
$\p(\A)$ satisfies $(\Pi_1)\sim(\Pi_4)$. Clearly, $\p(E)$ satisfies $(\Pi_1)$, $(\Pi_2)$ and
$(\Pi_4)$. We now show that $\p(E)$ satisfies $(\Pi_3)$.

Let $p,\,q\in\p(E)$. Then there exists a projection $f\in\p(\A)$, such that $f\sim \pi(p)$,
$f<\pi(q)$ and $\pi(q)-f\in\p(\A)$, that is, there is a partial isometry $v\in \A$ such that
$f=v{v}^{*}<\pi(q)$ and $\pi(p)={v}^{*}v$. Thus, there are $u\in E$ and a projection
$r\in p\B p$ such that $p-r=u^*u$, $uu^*\le q$ and $\pi(u)=v$ by Lemma \ref{ly2}.
Note that $q-uu^*\not\in\B$ and $(q-uu^*)\B(q-uu^*)\not=\{0\}$ ($\B$ is an essential ideal).
Then by Lemma \ref{ly3}, there is $w_0\in\B$ such that $r=w_0^*w_0$, $w_0w_0^*\in
(q-uu^*)\B(q-uu^*)$. Put
$\hat u=u+w_0$. Then $p=\hat u^*\hat u$, $\hat u\hat u^*\le q$ and $\pi(q-\hat u\hat u^*)=
\pi(q)-f\not=0$, i.e., $q-\hat u\hat u^*\in\p(E)$.
\end{proof}

\section{$K_{1}$-group of the extension algebra}

Recall from \cite{Xue} that a unital $C^*$--algebra $\E$ has $1$--cancellation, if a projection
$p\in\M_2(\E)$ satisfies $\diag(p,1_k)\sim\diag(p_1,1_k)$ for some $k$, then $p\sim p_1$ in
$\M_2(\E)$, where $p_1=\diag(1,0)$. If $\E$ has no unit and $\E^+$ has $1$--cancellation, we
say $\E$ has $1$--cancellation. It is known that when $\B$ has $1$--cancellation, we have
following exact sequence of groups:
\begin{equation}\label{ex}
U(\B)\stackrel{j_*}{\longrightarrow}U(E)\stackrel{\pi_*}{\longrightarrow}
U(\A)\stackrel{\eta}{\longrightarrow} K_{0}(\B)
\end{equation}
(cf. \cite[lemma 2.2]{Xue}), where $j_*$ (resp. $\pi$) is the induced homomorphism of the
inclusion $j\colon\B\rightarrow E$ (resp. $\pi$) on $U(\B)$ (resp. $U(E)$),
$\eta=\partial_0\circ i_\A$ and $\partial_0\colon K_1(\A)\rightarrow K_0(\B)$ is the index map.

Since, in general, we have the exact sequence of groups
$$
U(\B)\stackrel{j_*}{\longrightarrow}U(E)\stackrel{\pi_*}{\longrightarrow}U(\A),
$$
(for $\pi(\U_0(E))=\U_0(\A)$), i.e., $U(\cdot)$ is a half--exact and homotopic invariant functor,
it follows from Proposition 21.4.1, Corollary 21.4.2 and Theorem 24.4.3 of \cite{Bl} that
the sequence of groups
\begin{equation}\label{exx}
U(S\A)\stackrel{\partial}{\longrightarrow}U(\B)\stackrel{j_*}{\longrightarrow}U(E)
\stackrel{\pi_*}{\longrightarrow}U(\A)
\end{equation}
is exact, where $\partial=e_*^{-1}\circ i_*$ and $e\colon\B\rightarrow C_\pi$ given by
$e(b)=(b,0)\in C_\pi$, $e_*$ is isomorphic and
$i\colon S\A\rightarrow C_\pi$ is defined by $i(g)=(0,g)$, here
$$
C_\pi=\{(x,f)\in E\oplus C_0([0,1),\A)\vert\,\pi(x)=f(0)\},\quad S\A=C_0((0,1),\A).
$$
We also have the exact sequence
\begin{equation}\label{eyy}
K_1(S\A)\stackrel{\partial}{\longrightarrow}K_1(\B)\stackrel{j_*}{\longrightarrow}K_1(E)
\stackrel{\pi_*}{\longrightarrow}K_1(\A).
\end{equation}
\begin{proposition}\label{mt1}
Suppose that $i_\A$, $i_\B$ are isomorphic and $i_{S\A}$ is surjective. Assume that $\B$ has
$1$--cancellation. Then $i_E$ is an isomorphism.
\end{proposition}
\begin{proof}
Combining (\ref{ex}), (\ref{exx}) with (\ref{eyy}), we have following diagram
\begin{equation}\label{ezz}
\begin{array}{ccccccccc}
U(S\A)&\stackrel{\partial}\longrightarrow&U(\B)\stackrel{j_{*}}\longrightarrow&U(E)
\stackrel{\pi_{*}}\longrightarrow&U(\A)\stackrel{\eta}\longrightarrow&K_{0}(\B)\\
\downarrow i_{S\A}&&\downarrow i_{\B}&\downarrow i_{E}&\downarrow i_{\A}&\parallel \\
K_{1}(S\A)&\stackrel{\partial}\longrightarrow&K_{1}(\B)\stackrel{j_{*}}\longrightarrow&
K_{1}(E)\stackrel{\pi_{*}}\longrightarrow&K_{1}(\A)\stackrel{\partial_0}\longrightarrow&K_{0}(\B)\\
\end{array},
\end{equation}
in which two rows are exact and
$$
\eta=\partial_0\circ i_\A,\quad \pi_*\circ i_E=i_\A\circ\pi_*,\quad j_*\circ i_\B=i_E\circ j_*.
$$
Since $e_*$ is isomorphic, it follows from the commutative diagram
\[\begin{array}{ccccc}
U(S\A)&\stackrel{i_{*}}\longrightarrow&U(C_{\pi})\stackrel{e_{*}}\longleftarrow&U(\B)\\
\downarrow i_{S\A}&&\downarrow i_{C_{\pi}}&\downarrow i_{\B}\\
K_{1}(S\A)&\stackrel{i_{*}}\longrightarrow&K_{1}(C_{\pi})\stackrel{e_{*}}\longleftarrow&K_{1}(\B)\\
\end{array}
\]
that $\partial\circ i_{S\A}=i_\B\circ\partial$. Thus, (\ref{ezz}) is a commutative diagram.
Using the Five--Lemma to (\ref{ezz}), we can obtain the assertion.
\end{proof}

For a $C^*$--algebra $\E$, let $\csr(\E)$ and $\gsr(\E)$ be the connected stable rank and
general stable rank of $\E$, respectively, defined in \cite{R1}. We summrize some properties of
these stable ranks as follows:
\begin{lemma}\label{yl2}
Let $\E$ be a $C^*$--algebra. Then
\begin{enumerate}
\item $\gsr(\E)\le\csr(\E)$ {\rm(cf. \cite{R1});}
\item $\csr(\E)\le 2$ when $\E$ is a stable $C^*$--algebra {\rm(cf. \cite[Theorem 3.12]{Sh});}
\item $\E$ has $1$--cancellation if $\gsr(\E)\le 2$ {\rm(cf. \cite{Xue});}
\item if $\csr(\E)\le 2$ and $\gsr(C(\mathbf{S}^1,\E))\le 2$, then $i_\E$ is isomorphic
{\rm(cf. \cite[Theorem 2.9]{R2}} or \rm{\cite[Corollary 2.2]{Xue})}.
\end{enumerate}
\end{lemma}
Now we present the main result of this section as follows:
\begin{theorem}\label{dl2}
Assume that $\A$ is a unital simple purely infinite $C^*$--algebra and $\B$ is a stable
$C^*$--algebra. Let $X$ be a compact Hausdorff space. Then $i_{C(X,E)}$ is an isomorphism.
\end{theorem}
\begin{proof}
If $\B$ is stable, then so is $C(Y,\B)$ for any compact Hausdorff space $Y$. Thus,
$\gsr(C(\mathbf{S}^1,C(X,\B)))\le 2$ and $\csr(C(X,\B))\le 2$ by Lemma \ref{yl2} (1) and (2).
So we get that $i_{C(X,\B)}$ is isomorphic by Lemma \ref{yl2} (4).

Since $\A$ is unital simple purely infinite, it follows from \cite[Corollary 3.1]{Xue} that $i_{C(X,\A)}$
and $i_{SC(X,\A)}$ are all surjective. Now we prove $i_{C(X,\A)}$ is injective by using some
methods appeared in \cite{RLL}.

Let $f\in\U(C(X,\A))$ with $i_{C(X,\A)}([f])=0$ in $K_1(C(X,\A))$. Let $p$ be a non--trivial
projection in $\A$. Then there exists $g\in\U(C(X,p\A p))$ such that $f$ is homotopic to $g+1-p$
by \cite[Lemma 2.7]{Zh}. Thus, there is a continuous
path $f_t\colon [0,1]\rightarrow\U(\M_{n+1}(C(X,\A)))$ such that $f_0=1_{n+1}$ and
$f_1=\diag(g+1-p,1_n)$ for some $n\ge 2$. Since $\M_{n+1}(\A)$ is purely infinite, we can find
a partial isometry $v=(v_{ij})\in\M_{n+1}(\A)$ such that $\diag(1-p,1_n)=v^*v$,
$vv^*\le\diag(1-p,0)$. Consequently, we get that
$$
v_{11}^*v_{11}=1-p,\ v^*_{1j}v_{1,j}=1,\ v^*_{1j}v_{1,i}=0,\ i\not=j,\
\sum^{n+1}_{i=1}v_{1i}v^*_{1i}\le 1-p.
$$
Set $v_1=p+v_{11}$, $v_i=v_{1i}$, $i=2,\cdots n+2$. Then $v_1,\cdots v_{n+1}$ are isometries in
$\A$ and $v_i^*v_j=0$, $i\not=j$, $s=\sum\limits^{n+1}_{i=1}v_iv_i^*$ is a projection. Put
$$
w_t(x)=(v_1,\cdots,v_{n+1})f_t(x)\begin{pmatrix}v_1^*\\ v_2^*\\ \vdots\\ v^*_{n+1}\end{pmatrix}
+1-s,\quad t\in[0,1],\ x\in X.
$$
It is easy to check that $w_t$ is a continuous path in $\U(\M_n(C(X,\A)))$ with $w_0=1$ and
$w_1=g+1-p$. Thus, $i_{C(X,\A)}$ is injective.

The final result follows from Proposition \ref{mt1}.
\end{proof}

Combining Theorem \ref{dl2} with standard argument in Algebraic Topology, we can get
\begin{corollary}
Let $\A$, $\B$ and $E$ be as in Theorem \ref{dl2}. Then
$$
\pi_n(\U(E))=\begin{cases} K_0(E)& \ n\
\text{odd}\\ K_1(E)&\ n\ \text{even}\end{cases}.
$$
\end{corollary}

\end{document}